\documentclass[11pt]{amsart}
\usepackage{amsmath, amsthm, amsfonts, amssymb}
\usepackage{mathrsfs}
\usepackage{cite}
\usepackage{graphicx, xcolor}
\usepackage[colorlinks]{hyperref}

\newtheorem{theorem}{Theorem}[section]
\newtheorem{lemma}[theorem]{Lemma}
\newtheorem{proposition}[theorem]{Proposition}
\newtheorem{corollary}[theorem]{Corollary}

\DeclareMathOperator{\Aut}{Aut}
\DeclareMathOperator{\Sz}{Sz}
\DeclareMathOperator{\Syl}{Syl}

\DeclareMathOperator{\PGL}{PGL}
\DeclareMathOperator{\PSL}{PSL}
\DeclareMathOperator{\SL}{SL}

\newcommand{\sub}{\leqslant}
\newcommand{\nsub}{\nleqslant}

\newcommand{\Z}{\mathbb{Z}}
\newcommand{\N}{\mathcal{N}}
\newcommand{\C}{\mathcal{ C}}
\newcommand{\cyc}[1]{\langle #1\rangle}
\newcommand{\GAP}{{\rm GAP}}
\newcommand{\Core}{{\rm Core}}
\newcommand{\nor}{\trianglelefteq}
\newcommand{\nno}{\ntrianglelefteq}
\newcommand{\bk}{\backslash}

\begin{document}
\title[\scriptsize{Groups whose non-normal subgroups are either nilpotent...}]{Groups whose non-normal subgroups are either nilpotent or minimal non-nilpotent}

\author[Nasrin Dastborhan, Hamid Mousavi]{Nasrin Dastborhan, Hamid Mousavi$^{\ast}$}

\address{{Department of Pure Mathematics, Faculty of Mathematical Sciences, University of Tabriz, Tabriz, Iran}}
\email{hmousavi@tabrizu.ac.ir}
\email{n.dastborhan@tabrizu.ac.ir}

\thanks{$^{\ast}$ Corresponding author}

\begin{abstract}
Let $\mathfrak{Nil}$ be the class of nilpotent groups and $G$ be a  group. We call $G$ a meta-$\mathfrak{Nil}$-Hamiltonian group if any of its non-$\mathfrak{Nil}$ subgroups is normal. Also, we call $G$ a para-$\mathfrak{Nil}$-Hamiltonian group if  $G$ is a non-$\mathfrak{Nil}$ group and every non-normal subgroup of $G$ is either a $\mathfrak{Nil}$-group or a minimal non-$\mathfrak{Nil}$ group. In this paper we investigate the class of finitely generated meta-$\mathfrak{Nil}$-Hamiltonian and para-$\mathfrak{Nil}$-Hamiltonian groups.
\end{abstract}
\maketitle
\section{Introduction}
Let $\mathfrak{X}$ be a class of groups. The group $G$ is said to be meta-$\mathfrak{X}$-Hamiltonian if any of its non-$\mathfrak{X}$ subgroups is normal. Also, we say that $G$ is para-$\mathfrak{X}$-Hamiltonian if $G$ is a non-$\mathfrak{X}$ group and every non-normal subgroup of $G$ is either an $\mathfrak{X}$-group or a minimal non-$\mathfrak{X}$-group. If the class $\mathfrak{X}$ is subgroup closed, then the non-$\mathfrak{X}$ condition is necessary.

For a class $\mathfrak{X}$  of groups, a group $G$ is said to be minimal non-$\mathfrak{X}$ if it is not an $\mathfrak{X}$-group but all its proper subgroups belong to $\mathfrak{X}$. Also, $G$ is called biminimal non-$\mathfrak{X}$ if it is neither an $\mathfrak{X}$-group nor a minimal non-$\mathfrak{X}$-group, but each proper subgroup of $G$ either belongs to $\mathfrak{X}$ or is a minimal non-$\mathfrak{X}$-group. Para-$\mathfrak{X}$-Hamiltonian groups are a natural extention of  biminimal non-$\mathfrak{X}$ groups.

If $\mathfrak{A}$ is the class of Abelian groups, then the class of meta-$\mathfrak{A}$-Hamiltonian is called metahamiltonian and the class of para-$\mathfrak{A}$-Hamiltonian groups is called parahamiltonian.
The data about metahamiltonian groups is not very precise. These groups were introduced by Romalis in 1962. Then Romalis and Seskin continued the study of these groups.
They proved that if a metahamiltonian group is soluble, then its derived length is at most three and its commutator subgroup is finite of prime power order\cite{RoSea1,RoSea2,RoSea3}. Then, these groups  attracted the attention of many researchers, as de Giovanni and others. For a number of propertises  pertaining metahamiltonian groups one can see \cite{BFT, FangAn}. 

Some researchers have extended the metahamiltonian groups in a different direction. Instead of changing the class of subgroups that are not normal, they changed the normality condition of non-abelian subgroups. In their study, \cite{KAS} explored the structure of groups whose non-abelian subgroups are subnormal. Similarly, \cite{BFT1} focused on groups whose subgroups are either abelian or pronormal. Lastly, \cite{BT} examined locally finite simple groups whose non-abelian subgroups are pronormal.

A group $G$ is called locally graded if every non-trivial finitely generated subgroup of $G$ contains a proper subgroup of finite index. The results of Romalis and Sesekin \cite{RoSea1,RoSea2,RoSea3} show that a locally graded metahamiltonian group is necessarily soluble. Therefore, it has a finite commutator subgroup. Also, an infinite locally graded group whose proper subgroups are metahamiltonian is metahamiltonian~\cite[Theorem 3.1]{FGM}. However, the metahamiltonian groups are not necessarily soluble. For example, Tarski groups, i.e. infinite simple groups whose proper non-trivial subgroups have prime order.

 Atlihan and de Giovanni showed that every locally graded biminimal non-abelian group is finite and its order is divisible by at most three prime numbers\cite[Theorem 1]{AtGio}. They also showed that $G/\Phi(G)\cong A_5$, for any finite insoluble biminimal non-abelian group $G$. We will show that such a group is necessarily isomorphic to $A_5$ (Corollary~\ref{C-3.8}). 

Parahamiltonian groups were first introduced by Atlihan and de Giovanni\cite{AtGio}.  They proved that the commutator subgroup of a locally graded parahamiltonian group is finite \cite[Lemma 5]{AtGio}. They also showed that a parahamiltonian group which is locally graded and insoluble is finite and $\pi(G)=\{2,3,5\}$. We will show that such a group is necessarily isomorphic to $A_5$ (Corollary~\ref{C-3.9}). 

Let $\mathfrak{Nil}$ be the class of nilpotent groups. The most natural case for metahamiltonian and parahamiltonian generalization are meta-$\mathfrak{Nil}$-Hamiltonian and para-$\mathfrak{Nil}$-Hamiltonian groups. In this paper we investigate the class of finitely generated meta-$\mathfrak{Nil}$-Hamilton and para-$\mathfrak{Nil}$-Hamilton groups. It will be proved that every finitely generated locally graded biminimal non-nilpotent group is finite and its order is divisible by at most three prime numbers and every finite insoluble biminimal non-nilpotent group is isomorphic to $A_5$. Also, any locally graded insoluble para-$\mathfrak{Nil}$-Hamiltonian group is isomorphic to $A_5$ or $\SL(2,5)$, if it has a finite minimal non-nilpotent subgroup. In particular, this applies to any para-$\mathfrak{Nil}$-Hamiltonian group that is insoluble and finite. In the case of  locally graded meta-$\mathfrak{Nil}$-Hamiltonian groups, we show that they are soluble. Some of our techniques are taken from \cite{AtGio}.

\section{Primary lemmas}
In this article, we will use finite non-abelian minimal simple groups and the structure  of subgroups of $\PSL(2,p)$ and the Suzuki groups $\Sz(q)$. So, in this section, we will discuss some important structure theorems.

\begin{theorem}\cite[Theorems 6.25 , 6.26]{Suzuki}\label{sub-psl}
Let $q$ be a power of the prime $p$. Then, a subgroup of $\PSL(2,q)$ is isomorphic to one of the following groups.
\begin{itemize}
\item[(i)] The dihedral groups of order $2(q\pm1)/d$ and their subgroups where $d=(2,q-1)$. 
\item[(ii)] A group $H$ of order $q(q-1)/d$ and its subgroups. A Sylow $p$-subgroups $P$ of $H$ is elementary abelian, $P\unlhd H$, and the factor group $H/P$ is a cyclic group of order $(q-1)/d$.
\item[(iii)] $A_4$ whenever $p\neq2$ or $n$ is even.
\item[(iv)] $S_4$ whenever $q^2\equiv 1 \pmod{16}$.
\item[(v)] $A_5$ whenever $q(q^2-1)\equiv 0\pmod{5}$.
\item[(vi)] $\PSL(2,r)$ where $r$ is a power of $p$ such that, $q=r^m$.
\item[(vii)] $\PGL(2,r)$ where $q=r^m$ is odd and $m$ is even.
\end{itemize}
\end{theorem}

\begin{theorem}\cite[Theorem XI.3.10]{Huppert}\label{Sz}
We put $q=2^{2m+1}$, $r=2^m$ and $G=Sz(q)$.
\begin{itemize}
\item[(a)] $G$ possesses cyclic Hall subgroups $U_1$ and $U_2$ of orders $q+2r+1$ and $q-2r+1$.
\item[(b)]  If $1\neq u \in U_{i}$, $\C_{G}(u)=U_{i}$. Also $|\mathcal{N}_G (U_i):U_i |=4$  and $\mathcal{N}_G (U_i )=\cyc  {U_i,t_i}$ where $u^{t_i}=u^q$ for all $u\in U_i$. In particular, $\mathcal{N}_G (U_i )$ is a Frobenius group with Frobenius kernel $U_i$.
\end{itemize}
\end{theorem}

\begin{theorem}\cite[Corollary 1]{Thompson}\label{min-sim}
Every finite minimal simple group is isomorphic to one of the following minimal simple groups: 
\begin{itemize}
\item[(a)] $\PSL(2,2^p)$, $p$ any prime. 
\item[(b)] $\PSL(2,3^p)$, $p$ any odd prime. 
\item[(c)] $\PSL(2,p)$, $p$ any prime exceeding $3$ such that $p^2+1\equiv 0 \pmod{5}$
\item[(d)] $\Sz(2^p)$, $p$ any odd prime. 
\item[(e)] $\PSL(3,3)$.
\end{itemize}
\end{theorem}

\begin{theorem}\cite[12.1.5]{Rob}
If $M$ is a maximal subgroup of a locally nilpotent group $G$, then $M$ is normal in $G$. Equivalently $G'\leqslant \Phi(G)$.
\end{theorem}

\section{The finite case}
Let $G$ be a finite insoluble group. If $G$ is a para-$\mathfrak{Nil}$-Hamiltonian group or a biminimal non-nilpotent group, then we show that in any case $G$ is isomorphic to $\SL(2,5)$ or $A_5$.

\begin{lemma}\label{L-3.1}
Let $G$ be a finite group such that $G/\Phi(G)\cong\Z_p\rtimes\Z_q$, where $p$ and $q$ are distinct primes. Then, for natural numbers $m$ and $n$, $G\cong\Z_{p^m}\rtimes\Z_{q^n}$.
\end{lemma}

\begin{proof}
Let $P\in\Syl_p(G)$ and $Q\in\Syl_q(G)$. Since $P\Phi(G)\nor G$, by Frattini argument, $P\nor G$, so  $\Phi(P)\sub P\cap\Phi(G)$. If $\Phi(P)\neq P\cap\Phi(G)$, then there is a maximal subgroup $M$ of $P$ such that $P=M(P\cap\Phi(G))$ and so $G=PQ=M(P\cap\Phi(G))Q$, thus $G=MQ$, a contradiction. Therefore $\Phi(P)=P\cap\Phi(G)$ and so $P/\Phi(P)\cong \dfrac{P\Phi(G)}{\Phi(G)}\cong \Z_p$, hence $P$ is cyclic.

Since $\frac{Q}{Q\cap\Phi(G)}\cong \Z_q$, $Q\cap\Phi(G)$ is maximal in $Q$, so $\Phi(Q)\sub Q\cap\Phi(G)$. Now for any maximal subgroup $M$ of $Q$, $PM$ is maximal in $G$. As
\[Q\cap\Phi(G)\sub Q\cap PM=(P\cap Q)M=M,\]
then $Q\cap\Phi(G)\sub\Phi(Q)$. Therefore,  $\Phi(Q)=Q\cap\Phi(G)$. So, similar to the previous case, $Q/\Phi(Q)\cong \Z_q$ and so $Q$ is cyclic. Hence, for some $n,m\in \mathbb{N}$,
\[G=PQ\cong \Z_{p^n}\rtimes \Z_{q^m}\;\&\;
 \Phi(G)=\Z_{p^{n-1}}\times \Z_{q^{m-1}}.\]
\end{proof}

\begin{lemma}\label{L-3.2}
Let $G$ be a finite group and $\Phi(G)\sub H\sub G$. If $H$ and $H/\Phi(G)$ are minimal non-nilpotent, then $\Phi(G)\sub\Phi(H)$.
\end{lemma}

\begin{proof}
Let $M$ be a maximal subgroup of $H$. If $\Phi(G)\nsub M$, then $H=M\Phi(G)$ and so $H/\Phi(G)\cong M/(M\cap\Phi(G))$ is nilpotent, which is a contradiction. So $\Phi(G)\sub\Phi(H)$.
\end{proof}

\begin{theorem}\label{T-3.3}
Let $G$ be a finite group such that $G/\Phi(G)\cong A_5$. If $G$ is  para-$\mathfrak{Nil}$-Hamiltonian, then $G\cong A_5$ or $G\cong\SL(2,5)$.
\end{theorem}

\begin{proof}
Suppose that $H$ and $K$ are two subgroups of $G$ such that $H/\Phi(G)\cong S_3$ and $K/\Phi(G)\cong D_{10}$. As $H$ and $K$ are not nilpotent so are minimal non-nilpotent. According to Lemma \ref{L-3.2}, $\Phi(G)$ is contained both in $\Phi(H)$ and in $\Phi(K)$. Since $H$ and $K$ are noncyclic, therefore, $\Phi(H)=\Phi(K)=\Phi(G)$. Suppose that $H=\cyc{x,t_1}$ and $K=\cyc{y,t_2}$ where $t_1$ and $t_2$ are $2$-elements, $|x|=3^m$ and $|y|=5^n$, where $m,n\in N$. According to Lemma ~\ref{L-3.1}, $\Phi(H)=\cyc{x^3,t_1^2}$, $\Phi(K)=\cyc{y^5,t_2^2}$. Therefore, $x^3=y^5=1$ and $\Phi(G)$ is a cyclic $2$-subgroup. Since, $G/\C_G(\Phi(G))$ is abelian, thus $G=\C_G(\Phi(G))$ and so $Z(G)=\Phi(G)$. As $G$ is perfect, the sequence
\[M(G)\to M(G/Z(G))\to Z(G)\to 1,\]
is exact by ~\cite[Theorem 2.5.6]{Karpilovsky}, where $M(G)$ is the Schur multiplier of $G$. As $M(A_5)\cong\Z_2$, then  $|Z(G)|=2$ and $|G|=120$. Among the three insoluble groups of order 120, only $\SL(2,5)$ satisfies all conditions.
\end{proof}

\begin{lemma}\label{L-3.4}
Let $G$ be a finite minimal simple group. If $G$ is biminimal non-nilpotent, then $G\cong A_5$.
\end{lemma}

\begin{proof}
Assume that $G\cong\PSL(2,2^p)$ with $p$ prime. According to Theorem~\ref{sub-psl}, $G$ has two dihedral subgroups of orders $2(2^p-1)$ and $2(2^p+1)$. These subgroups are minimal non-nilpotent if both $2^p-1$ and $2^p+1$ are prime. Since $2^p+1$ is prime whenever $p=2$, therefore $G\cong\PSL(2,4)\cong A_5$ .

Suppose that $G\cong\PSL(2,3^p)$ for an odd prime $p$. Similar to the previous case, $G$ has two dihedral subgroups of orders $3^p-1$ and $3^p+1$. So both $(3^p\pm 1)/2$ are either prime or a power of $2$ greater than $8$. Let  
$(3^p+1)/2=2^{\ell}$ where $\ell\geq 3$. So $q=3^p=2^{\ell+1}-1$. Therefore 
$$q^2=3^{2p}=2^{2\ell +2}-2^{\ell+2}+1\equiv 1\pmod{16}.$$ 
Thus by Theorem ~\ref{sub-psl}(iv), $G$ has a subgroup isomorphic to $S_4$ that is neither nilpotent nor minimal non-nilpotent. Similarly, if $3^p-1$ is a power of $2$, again we will reach to contradiction caused by the existence of a subgroup isomorphic to $S_4$. So both $(3^p\pm 1)/2$ are prime. According to the relation 
$$(3^p+1)/2=(3^p-1)/2+1$$
necessarily $(3^p-1)/2=2$, a contradiction.

Now suppose that $G\cong\PSL(2,p)$ for a prime $p$. Similar to the above argument, $p\pm 1$ is not a power of $2$ greater than $8$, so $(p\pm 1)/2$ is necessarily prime. Since $(p+1)/2=(p-1)/2+1$, $p-1=4$ and so $p=5$, which contradicts $p\neq 5$ by Theorem~\ref{min-sim}(c). 

Using the $\GAP$  software~\cite{GAP}, it is determined that $G\not\cong\PSL(3,3)$. Finally assume that $G\cong\Sz(2^p)$, where $p$ is an odd prime. By Theorem~\ref{Sz}, $G$ has two cyclic Hall subgroups $U_1$ and $U_2$ of odd orders $n_1=2^p+2^{m+1}+1$ and $n_2=2^p-2^{m+1}+1$, respectively, where $m=(p-1)/2$. Also, $\N_G(U_i)$ for $i=1,2$ is a Frobenius group with kernel $U_i$, which is non-nilpotent. (The center of a Frobenius group is trivial). Since $|\N_G(U_i): U_i|=4$, then $\N_G(U_i)$ has a dihedral subgroup of order $2n_i$, which is non-nilpotent. Therefore, $\N_G(U_i)$ is not a minimal non-nilpotent. So $G\not\cong \Sz(2^p)$.
\end{proof}

\begin{lemma}\label{L-3.5}
Let $G$ be a finite minimal insoluble para-$\mathfrak{Nil}$-Hamiltonian group. Then, either $G\cong A_5$ or $G\cong \SL(2,5)$.
\end{lemma}

\begin{proof}
Since $G$ is minimal insoluble, $G/\Phi(G)$ is a minimal simple group. Therefore by Lemma~\ref{L-3.4} $G/\Phi(G)\cong A_5$. Now the statement follows from Theorem~\ref{T-3.3}
\end{proof}

\begin{theorem}\label{T-3.6}
Let $G$ be a finite insoluble para-$\mathfrak{Nil}$-Hamiltonian group. Then, $G\cong A_5$ or $\SL(2,5)$. In particular, the unique prime divisors of the order of G are $2$, $3$ and $5$.
\end{theorem}

\begin{proof}
It is clear that every insoluble subgroup of $G$ is normal. Let $N$ be a minimal insoluble normal subgroup of $G$. According to Lemma~\ref{L-3.5}, $N$ is isomorphic to $A_5$ or $\SL(2,5)$.

First, suppose that $N\cong A_5$. If $\C_G(N)\neq 1$, as $N\cap\C_G(N)=1$, the group $G$ has a non-normal subgroup isomorphic to $A_4\times\C_G(N)$ which is not minimal non-nilpotent, a contradiction. Hence, $\C_G(N)=1$, and $N\sub G\sub S_5$ by $N/C$-theorem. Therefore $G=N\cong A_5$.

Now assume that $N\cong\SL(2,5)$.  If $Z=Z(N)\neq \C_G(N)$, then $G/Z$ has a subgroup isomorphic to $A_5\times \C_G(N)/Z$, which leads to a contradiction similar to the previous case. Therefore, $Z(G)=\C_G(N)=Z(N)$, for $|Z(N)|=2$. Using $N/C$-theorem, $G/Z$ is isomorphic to a subgroup of $\Aut(N)\cong S_5$. As $G/Z\not\cong S_5$ and has a subgroup isomorphic to $A_5$, thus $G/Z\cong A_5$. Hence $|G|=|N|$ and so $G\cong\SL(2,5)$. 
\end{proof}

\begin{corollary}\label{C-3.7}
Let $G$ be a finite insoluble group. If $G$ is a biminimal non-nilpotent group, then $G$ is isomorphic to $\SL(2,5)$ or $A_{5}$.
\end{corollary}

Since $Q_8\sub\SL(2,3)\nno\SL(2,5)$ and $\SL(2,3)$  is not minimal non-abelian, therefore Lemma 1 and Theorem 2 of~\cite{AtGio} can be improved as follows.

\begin{corollary}\label{C-3.8}
Let $G$ be a finite insoluble biminimal non-abelian group. Then, $G\cong A_5$.
\end{corollary}

\begin{corollary}\label{C-3.9}
Let $G$ be a locally graded insoluble parahamiltonian group. Then, $G\cong A_5$.
\end{corollary}

\begin{proof}
According to \cite[Theorem 2]{AtGio}, $G$ is a finite insoluble group. Now, the statement holds by Theorem~\ref{T-3.6}.
\end{proof}

A subgroup $H$ of a group $G$ is called permutable (or quasinormal) if for all subgroups $K$ of $G$, $HK=KH$.
Obviously any normal subgroup is permutable. By \cite[Lemma 2.3]{ADE1}, if all non-nilpotent subgroups of finite group $G$ are permutable, then $G$ is soluble. So we can conclude that  finite meta-$\mathfrak{Nil}$-Hamiltonian groups are soluble.

\begin{theorem}\label{T-3.10}
Let $G$ be a finite meta-$\mathfrak{Nil}$-Hamiltonian group. Then $G$ is soluble.
\end{theorem}
Although the proof for the above theorem is not necessary, we can provide a straightforward proof.

Since every finite minimal non-nilpotent group is soluble, it can be assumed that $G$ has a proper non-nilpotent subgroup. Suppose that $X\sub G$ is non-nilpotent. So $G/X$ is a finite Dedekind group and $\gamma_3(G)\sub X$. Therefore, $\gamma_3(G)$ is contained in the intersection of all non-nilpotent subgroups of $G$. Since $\gamma_3(G)$ is nilpotent or minimal non-nilpotent, then it is soluble, which will cause $G$ to be soluble.

\section{Para-$\mathfrak{Nil}$-Hamiltonian and biminimal non-nilpotent groups}

Let $G$ be a group and \[G=\gamma_1(G)\geq \gamma_2(G)\geq\cdots\geq\gamma_n(G)\geq\cdots,\] be its lower central series.The intersection $\bigcap_{i\geq 1}\gamma_i(G)$ is denoted by $\gamma_{\infty}(G)$.

Minimal non-nilpotent groups and especially groups which are not nilpotent of class $k$ but whose proper subgroups are nilpotent of class $k$, were studied by Newman and Wiegold~\cite{NeWi}. The following two important theorems are about finitely generated non-nilpotent groups.

\begin{theorem}\cite[Lemma 3.2]{NeWi}\label{NW1}
Let G be a finitely generated non-nilpotent group all of whose proper subgroups are locally nilpotent. If $G/\gamma_{\infty}(G)$ is non-trivial, then $G$ is finite.
\end{theorem}

\begin{theorem}\cite[Theorem 3.3]{NeWi}\label{NW2}
 The Frattini factor group of an infinite finitely generated minimal non-nilpotent group is non-abelian and simple.
\end{theorem}

If the finitely generated minimal non-nilpotent group $G$ has a subgroup of finite index, then it has a normal subgroup $N$ of finite index. Since  $G/N$ is finite and soluble, therefore $\gamma_2(G)$ will be a proper subgroup of $G$ and then $G$ will be finite by Theorem~\ref{NW1}. Then, the infinite  finitely generated minimal non-nilpotent group has no subgroups of finite index, therefore, it is not locally graded. Thus, we will have the following result.

\begin{corollary}\label{C-4.3}
If a finitely generated minimal non-nilpotent group has a subgroup of finite index, then it is finite. In particular, every locally graded finitely generated minimal non-nilpotent group is finite.
\end{corollary}

Smith ~\cite[Theorem 1]{Smith} showed that {\it if every proper subgroup of a torsion free locally nilpotent group is nilpotent, then the group must be finite}. Also he proved that, {\it if $G$ is a soluble minimal non-nilpotent group that has no maximal subgroup, then $G$ has no proper subgroup of finite index} ~\cite[Theorem 3.1-v]{Smith}.

\begin{lemma}\label{L-4.4}
Let $G$ be a finitely generated locally graded group. If $G$ is infinite, then the finite residual $J$ of $G$ has infinite index. 
\end{lemma}

\begin{proof}
Assume that $J$ has finite index in $G$. Then $J$ is finitely generated and hence it contains a proper subgroup $K$ such that $|G: K|$ is finite. Therefore, $|G:\Core_{G}(K)|$ is also finite and so $ J \sub\Core_{G}(K)\sub K $ , which is a contradiction.
\end{proof} 

The following lemma is a generalization of~\cite[Lemma 4]{AtGio} by combining the parts of the proofs of~\cite[Lemma 4, Theorem 2]{AtGio}.

\begin{lemma}\label{closer}
Let $\mathfrak{X}$ be a class of groups that is subgroup closed and let $G$ be a para-$\mathfrak{X}$-Hamiltonian group. If $X\sub G$ which neither is normal nor is an $\mathfrak{X}$-group, then $X$ is maximal in its normal closure $X^G$. Furthermore, if $X$ is finite and $G$ is locally graded, then

\begin{itemize}
\item[(i)] $X^G$ is finite;
\item[(ii)] either $X\nor X^G$ or $G$ is finite.
\end{itemize}
\end{lemma}
\begin{proof}
Let $X< Y\sub G$. If $Y\nno G$, then either $Y\in\mathfrak{X}$ or $Y$ is minimal non-$\mathfrak{X}$ which implies that $X\in\mathfrak{X}$, a contradiction. Therefore, $Y\nor G$ and hence $X^G\sub Y$. Then $X$ is maximal in $X^G$.

(i) Now suppose that $X$ is finite and that $G$ is locally graded. Then $X^G$ is finitely generated, for $X$ is maximal in $X^G$. Let $Y$ be any normal subgroup of finite index of $X^G$. If $Y\sub X$ then $X^G$ is finite. Otherwise $X^G=YX$ and hence $|X^G: Y|\leq |X|$. Therefore, the finite residual of $X^G$ has finite index, then $X^G$ is finite by Lemma ~\ref{L-4.4}.

(ii) Let $X$ be a proper subgroup of $\N_G(X)$. Then $\N_G(X)\nor G$. Therefore $X^G\sub\N_G(X)$  and so $X\nor X^G$. Suppose that $X=\N_G(X)$, as $X^G$ is finite, then $X$ has finitely many conjugates in $G$. Hence $|G:\N_G(X)|$ is finite, and so $G$ is finite.
\end{proof}

If a group $G$ satisfies one of the following conditions, then its commutator subgroup $G'$ is finite:
\begin{itemize}
\item[(i)] $|H^G : H|$ is finite, for every subgroup $H$ of $G$~\cite[Theorem 5.70]{DKS}.
\item[(ii)] $G$ is a locally graded and metahamiltonian~\cite[Theorem 2.1]{FGM}.
\item[(iii)] $G$ is an insoluble locally graded parahamiltonian~\cite[Theorem 2.1]{AtGio}.
\end{itemize}
In case (iii), it follows from Corollary~\ref{C-3.9} that $G$ is isomorphic to $A_5$.
In the following theorem, we extend the finiteness of $G'$ with an additional condition for para-$\mathfrak{Nil}$-Hamiltonian groups. 

\begin{theorem}\label{T-4.6}
Let $G$ be a para-$\mathfrak{Nil}$-Hamiltonian group. If $G$ has a finitely generated minimal non-nilpotent subgroup $X$ whose normal closure $X^G$ is locally graded, then $G'$ is finite. Furthermore, if  either $X\nor G$ or  $G$ is infinite, then $G$ is soluble.
\end{theorem}
\begin{proof}
According to Corollary ~\ref{C-4.3}, $X$ is finite and hence is soluble. If $X\nor G$, then every subgroup of $G/X$ is normal, hence $G$ is soluble. Since $(G/X)'\cong G'/(G'\cap X)$ is finite, then $G'$ is finite. Now suppose that $X\nno G$. As $X$ is maximal in $X^G$, according to Lemma ~\ref{closer}, $X^G$ is finite. Since every subgroup of $G/X^G$ is normal, $(G/X^G)'$ is finite and hence $G'$ is finite. Suppose that $G$ is infinite. According to Lemma ~\ref{closer}, $X\nor X^G$. Therefore, $X^G$ is soluble and hence $G$ is solvable.
\end{proof}

\begin{corollary}\label{C-4.7}
Let $G$ be an insoluble para-$\mathfrak{Nil}$-Hamiltonian group. If $G$ has a finitely generated minimal non-nilpotent subgroup whose normal closure is locally graded, then $G\cong A_5$ or $G\cong\SL(2,5)$.
\end{corollary}
\begin{proof}
Let $X$ be a finitely generated minimal non-nilpotent subgroup of $G$. According to Theorem ~\ref{T-4.6}, $X\nno G$ and $G$ is finite. Hence the statement follows by Theorems ~\ref{T-3.6}.
\end{proof}
Every locally graded biminimal non-abelian group is finite and its order is divisible by at most three prime numbers~\cite[Theorem 1]{AtGio}. The similar result holds for finitely generated locally graded biminimal non-nilpotent groups.

\begin{theorem}\label{T-4.8}
Let $G$ be a finitely generated locally graded biminimal non-nilpotent group. Then $G$ is finite, and its order is divisible by at most three prime numbers. 
\end{theorem}
\begin{proof}
Since $G$ is locally graded, it has a proper normal subgroup $N$ of finite index. It is obvious that $N$ is finitely generated. If $N$ is non-nilpotent, it is finite by Corollary ~\ref{C-4.3}, therefore the statement holds. So assume that $N$ is nilpotent. Then $N$ is supersoluble by~\cite[5.4.6(ii)]{Rob}. Since $G$ is a biminimal non-nilpotent group, it has a proper subgroup $M$ that is minimal non-nilpotent. Also, $M$ is a maximal subgroup of $G$. We show that $M$ has finite index, thus it is finitely generated, and according to Corollary ~\ref{C-4.3}, $G$ will be finite.

It can be assumed that $N\nsub M$, then $G=NM$ and hence $N\cap M\nor G$. Let $H$ be the maximal subgroup of $N$ containing $N\cap M$. Since $N$ is supersoluble, $H$ has finite index in $N$ \cite[5.4.7]{Rob}, also  $N\cap M\sub H_G$, where $H_G=\Core_G(H)$. If $N\cap M\neq H_G$ then $G=H_GM$. Now for each $x\in N\bk H$, there are $h\in H_G$ and $m\in M$, where $x=hm$. Then $h^{-1}x=m\in N\cap M\sub H$, which is a contradiction. Therefore, $N\cap M=H_G$ and $M$ has finite index. 

The statement about prime divisors holds for any insoluble group by Corollary ~\ref{C-3.7}. So suppose that $G$ is a soluble group and $H$ is a non-nilpotent subgroup of $G$. Therefore, $H$ is minimal non-nilpotent and is maximal in $G$. As $|H|$ has two prime factors and $|G:H|$ is a power of a prime number, so the order of $G$ has at most three prime factors.   
\end{proof}

\begin{corollary}\label{C-4.9}
Let $G$ be a locally graded biminimal non-nilpotent group. If $G$ has a finitely generated non-nilpotent subgroup, then $G$ is finite.
\end{corollary}
\begin{proof}
Let $X$ be a finitely generated non-nilpotent subgroup of $G$. According to Theorem ~\ref{T-4.8}, it can be assumed that $X$ is a proper subgroup of $G$. Therefore, $X$ is minimal non-nilpotent and so it is a maximal subgroup of $G$. By Corollary \ref{C-4.3}, $X$ is finite, therefore $G$ is finitely generated. Now $G$ is finite by Theorem ~\ref{T-4.8}.
\end{proof}

Suppose that $G$ is a locally graded infinite group whose subgroups are either nilpotent or minimal non-nilpotent. According to Corollary ~\ref{C-4.9}, if $G$ is not nilpotent, then it is neither finite nor does it have any non-nilpotent subgroups that are finitely generated. Therefore, it is locally nilpotent.

\begin{corollary}\label{C-4.10}
Let $G$ be an infinite locally graded group whose subgroups are either nilpotent or minimal non-nilpotent. Then $G$ is locally nilpotent.
\end{corollary}

\section{Meta-$\mathfrak{Nil}$-Hamiltonian groups}

In this section, we will discuss some properties of meta-$\mathfrak{Nil}$-Hamiltonian groups. Firstly, we will prove that if a finitely generated meta-$\mathfrak{Nil}$-Hamiltonian group is non-nilpotent but has a nilpotent subgroup of finite index, then it is soluble with a finite derived subgroup. Secondly, we will show that any locally graded meta-$\mathfrak{Nil}$-Hamiltonian group is soluble.

\begin{lemma}\label{L-5.1}
Let $G$ be a biminimal non-nilpotent locally nilpotent group. Then $G$ is meta-$\mathfrak{Nil}$-Hamiltonian.
\end{lemma}

\begin{proof}
If $X$ is a proper non-nilpotent subgroup of $G$, then $X$ is a maximal subgroup of $G$. Hence $\Phi(G) \sub X$. On the other hand, since $G$ is a locally nilpotent, $ G' \sub X$ and consequently $X$ is normal.
\end{proof}

\begin{lemma}\label{L-5.2}
Let $G$ be a non-nilpotent finitely generated meta-$\mathfrak{Nil}$-Hamiltonian group. If every proper subgroup of $G$ is locally nilpotent, then either $G$ is finite or $G/\Phi(G)$ is an infinite simple group.
\end{lemma}

\begin{proof}
If $G$ is an infinite minimal non-nilpotent group, then $G/\Phi(G)$ is non-abelian and simple according to Theorem~\ref{NW2}. Otherwise $G$ has a non-nilpotent proper subgroup $X$. Since $G/X$ is a Dedekind group, then $\gamma_3(G)\sub X$. Therefore, $G/\gamma_{\infty}(G)$ is non-trivial. Now the statement  follows from Theorem~\ref{NW1}.
\end{proof}

Therefore, every non-nilpotent finitely generated locally nilpotent group whose non-nilpotent subgroups are normal, is either finite or infinite insoluble. 

The following lemma holds with a completely similar proof from~\cite[Lemma 3.1]{FGM1}.

\begin{lemma}\label{L-5.3}
Let $\mathfrak{X}$ be a class of groups that is subgroup closed. Assume that any non-$\mathfrak{X}$ subgroup of $G$ contains a finitly generated non-$\mathfrak{X}$ subgroup. If every finitely generated subgroup of $G$ is meta-$\mathfrak{X}$-Hamiltonian, then $G$ is meta-$\mathfrak{X}$-Hamiltonian.
\end{lemma}

\begin{proof}
Suppose that $X$ is a non-$\mathfrak{X}$ subgroup of $G$ and $\mathfrak{N}$ is the set of finitely generated non-$\mathfrak{X}$ subgroups of $X$, by assumption $\mathfrak{N}$ is non-empty. Now for any $Y\in\mathfrak{N}$, $\cyc{Y,g}$ is finitely generated for all $g\in G$, so $Y^g=Y$, then $Y\nor G$. Also for any $g\in X$,  $\cyc{Y, g}\in\mathfrak{N}$, so $X=\cyc{Y\,|\, Y\in\mathfrak{N}}\nor G$.
\end{proof}

\begin{lemma}\label{L-5.4}
Let $\mathfrak{X}$ be a class of groups that is subgroup closed and $G$ be a meta-$\mathfrak{X}$-Hamiltonian group. Suppose that $J$ is finite residual of $G$, $\mathfrak{N}$ is the class of non-$\mathfrak{X}$ subgroups of $G$ and $I=\bigcap_{X\in\mathfrak{N}}X$.
\begin{itemize}
\item[(i)] $\gamma_3(G)\sub I$ and $I$ is either an $\mathfrak{X}$-group or a minimal non-$\mathfrak{X}$ group.
\item[(ii)] If $G$ is not an $\mathfrak{X}$-by-finite, then $\gamma_3(G)\sub I\sub J$.
\end{itemize}
\end{lemma}

\begin{proof}
(i) If $\mathfrak{N}=\emptyset$, then $I=G$ and there is nothing left to prove. Assume that $\mathfrak{N}\neq\emptyset$ and $X\in\mathfrak{N}$. Since $G/X$ is a Dedekind group, $\gamma_3(G)\sub X$. Therefore, $\gamma_3(G)\sub I$. Now suppose that $X\sub I$ is a non-$\mathfrak{X}$-group, then $I=X$ and so $I$ is a minimal non-$\mathfrak{X}$-group.

(ii) Suppose that $G$ is not an $\mathfrak{X}$-by-finite group. Let $\mathfrak{F}$ be the class of subgroups of finite index of $G$. Therefore, none of the members of $\mathfrak{F}$ is an $\mathfrak{X}$-group, so $\mathfrak{F}\sub\mathfrak{N}$. Thus,
\[\gamma_3(G)\sub I\sub\bigcap_{N\in\mathfrak{F}} N=J.\] 
\end{proof}

Notice here that, according to Lemma~\ref{L-5.4} (ii), non-nilpotent meta-$\mathfrak{Nil}$-Hamiltonian groups are not residually finite.

\begin{proposition}\label{P-5.5}
Let $G$ be a non-nilpotent finitely generated meta-$\mathfrak{Nil}$-Hamiltonian group. If $G$ has a torsion-free nilpotent subgroup $H$ of finite index, then $G\cong H_G\times K$, where $K$ is finite non-nilpotent soluble and $H_G$ is a free abelian group of finite rank. In particular, $G$ is soluble and $G'$ is finite.
\end{proposition}

\begin{proof}
Suppose that $H$ is a torsion-free nilpotent subgroup of finite index in $G$. Hence $|G: H_G|$ is also finite. In particular, $N=H_G$ is finitely generated, torsion-free nilpotent. Suppose that $p$ and $q$ are primes such that $q>p>|G/N|$. Since $|N:N^{pq}|$ is finite, $|G/N^{pq}|$ is also finite. Since $|\bar{G}: \bar{N}|=|G/N|<p$, where $\bar{G}=G/N^{pq}$, then $|\bar{G}: \bar{N}|$ and $|\bar{N}|$ are relatively prime, $\bar{G}\cong \bar{N}\rtimes \bar{K}$ where $\bar{K}\cong G/N$,  by Schur-Zassenhaus Theorem.

Since Sylow subgroups of $\bar{N}$ are characteristic, ${\bar{N}}^p\bar{K}$ and ${\bar{N}}^q\bar{K}$ are proper subgroups of $\bar{G}$. If, for example, ${\bar{N}}^p\bar{K}$ is nilpotent, then $K\cap N^p$ and $ [K, N^p]$ are both contained in $N^{pq}$. Since $N$ is a finitely generated torsion-free nilpotent group, according to \cite[5.2.21]{Rob}, $\bigcap_{q>p}N^{pq}\sub\bigcap_{q>p}N^q=1$. Therefore $N^p\cap K=[K,N^p]=1$ and hence $G\cong N^p\times K$ is nilpotent, because $\bar{K}\cong K$. This is a contradiction. Therefore,  ${\bar{N}}^p\bar{K}$ and ${\bar{N}}^q\bar{K}$ both are non-nilpotent and so are normal in $\bar{G}$. Hence $\bar{K}\nor\bar{G}$, thus $\bar{K}$ acts trivially on $\bar{N}$. Then  both $K\cap N$ and $ [K, N]$ are contained in $N^{pq}$ and hence in  $N^{q}$. So again, we will have $G\cong N\times K$. Since $G/K\cong N$ is torsion-free Dedekind group, it is abelian. Therefore, $N$ is a free abelian group of finite rank. By Theorem~\ref{T-3.10}, $K$ is soluble.
\end{proof}

The group $\Z\times S_3$ is a simplest non-nilpotent example of  Proposition~\ref{P-5.5}.

\begin{theorem}\label{T-5.6}
Let $G$ be a finitely generated non-nilpotent meta-$\mathfrak{Nil}$-Hamiltonian group and $H\sub G$ be nilpotent of finite index. Then $G$ is soluble and $G'$ is finite.
\end{theorem}

\begin{proof}
According to Theorem~\ref{T-3.10}, we can assume that $G$ is infinite. As $|G : H|$ is finite, $H_G$ is a finitely generated nilpotent group, thus $G$ is polycyclic-by-finite. Hence $T_H$ the torsion subgroup of $H_G$ is finite by~\cite[2.49(ii)]{DKS}. As $H_G/T_H\neq 1$ is a torsion free nilpotent subgroup of $G/T_H$, if $G/T_H$ is non-nilpotent, according to Proposition~\ref{P-5.5}, $(G/T_H)'$ is a finite soluble group. Therefore $G'$ is a finite soluble group and so $G$ is soluble.
Assume that $G/T_H$ is nilpotent. Again $G$ is soluble. Since $G$ is non-nilpotent, there exists smallest subgroup $1\neq T_1$ of $T_H$ such that $G/T_1$ is nilpotent and $T_1\nsub\Phi(G)$ by~\cite[4.4.5(ii)]{LeRob}.
Let $M$ be a maximal subgroup of $G$ such that $G=T_1M$. Assume that $T_M$ is the torsion subgroup of $M_G$. If $G/T_M$ is nilpotent, then $G/T_1\cap T_M$ is nilpotent, which is a contract to chosen of $T_1$. So $G/T_M$ is non-nilpotent and $(H_G\cap M_G)T_M/T_M$ is a torsion free nilpotent subgroup of finite index of $G/T_M$. Now by Proposition~\ref{P-5.5}, $(G/T_M)'$ is finite, therefore $G'$ is finite.
\end{proof}

If a meta-$\mathfrak{Nil}$-Hamiltonian group $G$ has a finite non-nilpotent subgroup $X$, since $G/X$ is a Dedekind group, then $G'$ is finite.

\begin{corollary}\label{C-5.7}
Let $G$ be a locally graded meta-$\mathfrak{Nil}$-Hamiltonian group. If $G$ has a finitely generated minimal non-nilpotent subgroup, then $G'$ is finite.
\end{corollary}

If $G$ is insoluble such that all whose non-nilpotent subgroups are permutable, then $G''$ is perfect by \cite[Lemma 2.1]{ADE1}. In the following propositions  with the meta-$\mathfrak{Nil}$-Hamiltonian assumption, we will have a stronger result for $G$.

\begin{proposition}\label{P-5.8}
Let $G$ be an insoluble meta-$\mathfrak{Nil}$-Hamiltonian group. Then the following statements are equivalent.
\begin{itemize}
\item[(i)] $G$ is perfect.
\item[(ii)] $G$ is minimal non-nilpotent.
\end{itemize}
\end{proposition}

\begin{proof}
(i)$\rightarrow$(ii) Assume that $G$ is perfect and $X\sub G$ is non-nilpotent. Since $G/X$ is Dedekind group we have the contradiction $G'\neq G$, if  $X\neq G$. Therefore, $G$ is minimal non-nilpotent.

(ii)$\rightarrow$(i) Suppose that $G$ is minimal non-nilpotent. Since  $G'$ is not soluble, so is not proper subgroup of $G$. Therefore $G=G'$.
\end{proof}

\begin{proposition}\label{P-5.8.5}
Let $G$ be an insoluble meta-$\mathfrak{Nil}$-Hamiltonian group. Then the following statements are equivalent.
\begin{itemize}
\item[(i)] $G$ is perfect with a maximal subgroup.
\item[(ii)] $G$ is infinite finitely generated and $G/\Phi(G)$ is non-abelian and simple. 
\end{itemize}
\end{proposition}

\begin{proof}
(i)$\rightarrow$(ii) By Proposition~\ref{P-5.8}, $G$ is minimal non-nilpotent. If $G$ is not finitely generated, by \cite[4.8]{NeWi}, $G'$ is abelian, a contradiction. Therefore $G$ is infinite finitely generated group and so, $G/\Phi(G)$ is an infinite and simple group by Theorem~\ref{NW2}.

(ii)$\rightarrow$(i) straightforward.
\end{proof}

\begin{proposition}\label{P-5.9}
Let $G$ be an insoluble meta-$\mathfrak{Nil}$-Hamiltonian group. If $G$ is not perfect, then $G''=\gamma_3(G)$ is a perfect minimal non-nilpotent group. Furthermore either $G''$ is finitely generated and its Frattini factor is an infinite and simple group or $G''$ is Fitting $p$-group for some prime $p$.
\end{proposition}

\begin{proof}
According to Proposition~\ref{P-5.8}, $G$ has a proper non-nilpotent subgroup. Therefore, $\gamma_3(G)$ is included in all non-nilpotent subgroups of $G$. Since $G''\sub\gamma_3(G)$ is non-nilpotent, $G''=\gamma_3(G)$ is minimal non-nilpotent by Lemma~\ref{L-5.4}. By reusing  Proposition~\ref{P-5.8}, $G''$ is perfect.
If $G$ is finitely generated, by Theorem~\ref{NW2}, the Frattini factor group of $G''$ is infinite simple group. Otherwise, if $G''$ has a maximal subgroup, by \cite[4.8]{NeWi}, is not perfect, a contradiction. Hence $G''$ dose not have any maximal subgroup and so is Fitting $p$-group for some prime $p$ by~\cite[Theorem 3.3 (i) and (ii)]{Smith}.
\end{proof}

 Atlihan and et al. proved that {\it if all non-nilpotent subgroups of locally graded group $G$ are permutable, then $G$ is soluble} (see \cite[Theorems A, B]{ADE1} and \cite[ Theorem 1.1]{ADE2}). Therefore a locally graded meta-$\mathfrak{Nil}$-Hamiltonian group is soluble. 
In the following lemma we have a simple proof  for solubility of locally graded meta-$\mathfrak{Nil}$-Hamiltonian groups.

\begin{lemma}\label{L-5.10}
Let $G$ be a locally graded meta-$\mathfrak{Nil}$-Hamiltonian group. Then $G$ is soluble.
\end{lemma}

\begin{proof}
Suppose that $G$ is insoluble, so it does not have a nilpotent subgroup of finite index (by using Theorem~\ref{T-3.10}). According to~\cite[Corollary 2.3]{MaGi}, $G$ has a soluble subgroup $M$ of finite index. Therefore, $M$ is non-nilpotent and so is normal, which contradicts the insolubility of $G$. 
\end{proof}

\begin{theorem}\label{T-5.11}
Let $G$ be a non-nilpotent finitely generated locally graded meta-$\mathfrak{Nil}$-Hamiltonian group. If $G'$ is infinite, $G$ has a nilpotent maximal subgroup of infinite order and index. In particular, the Frattini subgroup and all finite normal subgroups of $G$ are nilpotent. According to the symbols of Lemma~\ref{L-5.4}, 
\begin{itemize}
\item[(i)] If $G/X$ is non-abelian for some non-nilpotent subgroup $X$, then $J$ is nilpotent. Therefore, if $J$ is non-nilpotent, $G'\sub I$; 
\item[(ii)] If $I$ is non-nilpotent, then $I$ is infinitely generated minimal non-nilpotent with a maximal subgroup, $G'\sub I$ and $I'=\gamma_n(G)=\gamma_{\infty}(G)$ is abelian for some $n$.  Furthermore $G'/\gamma_{\infty}(G)$ is finite cyclic $p$-group for some prime $p$.
\end{itemize}
\end{theorem}

\begin{proof}
As $G$ is soluble by Lemma~\ref{L-5.10} and is non-nilpotent finitely generated group by hypotsis, so it has a non-normal nilpotent maximal subgroup $N$ (by  \cite[4.4.5(ii)]{LeRob}). By Theorem~\ref{T-5.6}, $N$ is not of finite index, so every finite normal subgroup of $G$ is contained in $N$ and so is nilpotent. If $N$ is finite, since for any subgroup $X$ of finite index, $X$ is non-nilpotent, so $G=XN$. Then $|G: X|\leq |N|$, thus $G$ has a finite number of  subgroups of finite index. Therefore, $J$ has finite index, which contradicts Lemma~\ref{L-4.4}.

(i) In this case, the factor group $G/J$ is non-abelian. If $J$ is non-nilpotent then $G/J$ is a finitely generated Dedekind torsion group, so it is finite which contradicts Lemma~\ref{L-4.4}. Therefore, $J$ is nilpotent.

(ii) According to Lemma~\ref{L-5.4}, $I$ is minimal non-nilpotent, so by (i), $G'\sub I$ and $I$ is not finitely generated by Corollary~\ref{C-4.3}. 

Assume that $I$ has no maximal subgroup, so by~\cite[Theorem 3.1]{Smith}, $I$ is a countable $p$-group for some prime $p$ and $I/I'\cong C_{p^\infty}$. If $I'\nsub N$, then $G=I'N$ and so $I=I'(I\cap N)$. By reuse of ~\cite[Theorem 3.1]{Smith}, $I=I\cap N$ and so $G'\sub N$, which contradicts the non-normality of $N$.
As $N\cap I$ is a proper subgroup of $I$, so $(N\cap I)/I'$ is a finite. Assume that $H/I'$ is a proper subgroup of $I/I'$ such that properly contains $(N\cap I)/I'$. Then $G=HN$ and $|G:N|\leq |H/I'|$ is finite which is a contradiction. Therefore $I$ has a maximal subgroup, so $I'$ is abelian and $I/I'$ is a finite cyclic $p$-group, for some prime $p$ (see \cite[4.7, 4.8]{NeWi}), thus $G''$ is abelian and $I'=\gamma_3(I)$ by~\cite[Corollary 2.12]{NeWi}. If $I'\sub N$, since $I/I'$ is finite, so $N$ is of finite index, which is a contradiction. So $I'\nsub N$  and $G/I'$ is nilpotent. Then $\gamma_n(G)\sub I'=\gamma_{\infty}(I)\sub\gamma_{\infty}(G)$ for some $n$, therefore $I'=\gamma_n(G)=\gamma_{\infty}(G)$ and so $G'/\gamma_{\infty}(G)$ is a finite cyclic $p$-group.
\end{proof}

\subsection*{Acknowledgement}
We dedicate this article to Professor Francesco de Giovanni, who conducted valuable research in promoting the theory of groups and trained many researchers.

\end{document}